\documentclass[12pt,leqno]{article}
\usepackage{amsfonts}
\usepackage{amsmath, amsthm, amsfonts, amssymb, color, ulem}
\usepackage{mathrsfs}
\newtheorem{thm}{Theorem}[section]

\newtheorem{lem}[thm]{Lemma}

\theoremstyle{definition}

\newcommand{\scr}[1]{\mathscr #1}
\definecolor{wco}{rgb}{0.5,0.2,0.3}
\renewcommand{\bar}{\overline}
\numberwithin{equation}{section}
\newtheorem{rem}{Remark}[section]

\newcommand{\ua}{\uparrow}

\renewcommand{\hat}{\widehat}
\renewcommand{\tilde}{\widetilde}

\title{{\bf Convergence rates of theta-method for neutral SDDEs under non-globally Lipschitz continuous coefficients\thanks{Supported by NSFC(No., 11561027, 11661039), NSF of Jiangxi(No., 20161BAB211018), Scientific Research Fund of Jiangxi Provincial Education Department(No., GJJ150444).}}
}

\author{
{\bf  Li Tan$^{a,b}$\, and \, Chenggui Yuan$^c$}\\
 \footnotesize{$^{a}$ School of Statistics, Jiangxi University of Finance and Economics, Nanchang, Jiangxi, 330013, P. R. China}\\
 \footnotesize{$^{b}$ Research Center of Applied Statistics, Jiangxi University of Finance and Economics,}\\
  \footnotesize{ Nanchang, Jiangxi, 330013, P. R. China}\\
\footnotesize{$^c$ Department of Mathematics, Swansea University, Swansea, SA2 8PP, U. K. }\\
C.Yuan@swansea.ac.uk}

\begin{document}
\def\R{\mathbb R}  \def\ff{\frac} \def\ss{\sqrt} \def\B{\mathbf
B}
\def\N{\mathbb N} \def\kk{\kappa} \def\m{{\bf m}}
\def\dd{\delta} \def\DD{\Delta} \def\vv{\varepsilon} \def\rr{\rho}
\def\<{\langle} \def\>{\rangle} \def\GG{\Gamma} \def\gg{\gamma}
  \def\nn{\nabla} \def\pp{\partial} \def\EE{\scr E}
\def\d{\text{\rm{d}}} \def\bb{\beta} \def\aa{\alpha} \def\D{\scr D}
  \def\si{\sigma} \def\ess{\text{\rm{ess}}}
\def\beg{\begin} \def\beq{\begin{equation}}  \def\F{\scr F}
\def\Ric{\text{\rm{Ric}}} \def\Hess{\text{\rm{Hess}}}
\def\e{\text{\rm{e}}} \def\ua{\underline a} \def\OO{\Omega}  \def\oo{\omega}
 \def\tt{\tilde} \def\Ric{\text{\rm{Ric}}}
\def\cut{\text{\rm{cut}}} \def\P{\mathbb P} \def\ifn{I_n(f^{\bigotimes n})}
\def\C{\scr C}      \def\aaa{\mathbf{r}}     \def\r{r}
\def\gap{\text{\rm{gap}}} \def\prr{\pi_{{\bf m},\varrho}}  \def\r{\mathbf r}
\def\Z{\mathbb Z} \def\vrr{\varrho} \def\ll{\lambda}
\def\L{\scr L}\def\Tt{\tt} \def\TT{\tt}\def\II{\mathbb I}
\def\i{{\rm in}}\def\Sect{{\rm Sect}}\def\E{\mathbb E} \def\H{\mathbb H}
\def\M{\scr M}\def\Q{\mathbb Q} \def\texto{\text{o}} \def\LL{\Lambda}
\def\Rank{{\rm Rank}} \def\B{\scr B} \def\i{{\rm i}} \def\HR{\hat{\R}^d}
\def\to{\rightarrow}\def\l{\ell}
\def\8{\infty}\def\Y{\mathbb{Y}}

\maketitle

\begin{abstract}
This paper is concerned with strong convergence and almost sure convergence for neutral stochastic differential delay equations under non-globally Lipschitz continuous coefficients. Convergence rates of $\theta$-EM schemes are given for these equations driven by Brownian motion and pure jumps respectively, where the drift terms satisfy locally one-sided Lipschitz conditions, and diffusion coefficients obey locally Lipschitz conditions, and the corresponding coefficients are highly nonlinear with respect to the delay terms.

\end{abstract}

 \vskip 2cm

\section{Introduction}
With the development of computer technology, numerical analyses have been witnessed rapid growth since most equations can not be solved explicitly. There is an extensive literature concerned with numerical solutions for stochastic differential equations (SDEs) and stochastic differential delay equations (SDDEs). In 1955, Maruyama \cite{m55} put forward Euler-Maruyama (EM) scheme for SDEs. After that, there is a strong interest in numerical methods to all kinds of differential equations. Gikhman and Skorokhod \cite{gs72} showed that under global Lipschitz and linear growth condition, EM scheme converges to exact solution with order 1/2 for SDEs, while for additive noise case, the convergence rate is 1. Kloeden and Platen \cite{kp92} also studied numerical methods under a global Lipschitz condition. However, the global condition sometimes is too strict. In order to cover a larger part of SDEs, Higham et al. \cite{hms02} studied strong convergence of Euler-type methods for nonlinear SDEs. They gave convergence rate of EM scheme and backward Euler scheme for SDEs under local Lipschitz and one-sided Lipschitz condition. Later, Mao and Sabanis \cite{ms03} showed that the EM scheme will converge to exact solutions for SDDEs under a local Lipschitz condition. Higham and Kloeden \cite{hk05} presented and analysed two implicit methods for It\^{o} SDEs with Poisson-driven jumps where coefficients satisfy local Lipschitz conditions. Bao and Yuan \cite{by13} investigated convergence rate of EM scheme for SDDEs, where the corresponding coefficients may be highly nonlinear with respect to the delay variables. There are also some other literature concerning with strong convergence of explicit and implicit Euler-type methods to SDEs or SDDEs under non-global Lipschitz conditions, see \cite{hjk12, hjk13, ms13, zwh15} and the reference therein.

Neutral SDDEs plays an important role in stochastic analysis. As to its numerical analysis, Wu and Mao \cite{wm08} examined numerical solutions of neutral stochastic functional differential equations and established the strong mean square convergence theory of EM scheme under local Lipschitz condition; Zhou \cite{z15} established a criterion on exponential stability of EM scheme and backward scheme to neutral SDEs; Zong and Huang \cite{zw16} concerned with $p$-th moment and almost sure exponential stability of the exact and EM-scheme solutions of neutral SDDEs; Ji et al. \cite{jby16} generalized the results of \cite{by13} to neutral SDDEs; Tan and Yuan \cite{ty16} proposed a tamed $\theta$-EM scheme and gave convergence rate for neutral SDDEs driven by Brownian motion and pure jumps under one-sided Lipschitz condition. Motivated by Bao and Yuan \cite{by13} and Zong et al. \cite{zwh15}, the drift and diffusion coefficients may be highly nonlinear with respect to delay variables. Will the $\theta$-EM scheme converges to exact solutions strongly and almost surely for neutral SDDEs if the drift terms satisfy locally one-sided Lipschitz condition and diffusion coefficients obey locally Lipschitz conditions, and the corresponding coefficients is highly nonlinear? In this paper, we shall give a positive answer when the corresponding coefficients are highly nonlinear with respect to the delay terms.

The rest of paper is organized as follows: in Section 2, strong convergence rate and almost sure convergence rate are given for neutral SDDEs driven by Brownian motion under non-globally Lipschitz condition, while in Section 3, the Brownian motion is replaced by pure jumps,  the convergence rates are also provided under similar conditions.

\section{Convergence Rates for Brownian Motion Case}
\subsection{Preliminaries}
Let $(\Omega,\mathscr{F},\{\mathscr{F}_t\}_{t\ge 0}, \mathbb{P})$ be a complete probability space with $\{\mathscr{F}_t\}_{t\ge 0}$ satisfying the usual conditions (i.e., it is right continuous and increasing while $\mathscr{F}_0$ contains all $\mathbb{P}$-null sets). $(\mathbb{R}^n, \langle\cdot,\cdot\rangle,|\cdot|)$ is an $n$-dimensional Euclidean space. Denote $\mathbb{R}^{n\times d}$ by the set of all $n\times d$ matrices $A$ with trace norm $\|A\|=\sqrt{\mbox{trace}(A^TA)}$, where $A^T$ is the transpose of matrix $A$. For a given $\tau\in(0,\infty)$, define the uniform norm $\|\zeta\|_\infty:=\sup_{-\tau\le\theta\le 0}|\zeta(\theta)|$ for $\zeta\in\mathcal{C}([-\tau,0];\mathbb{R}^n)$ which denotes all continuous functions from $[-\tau,0]$ to $\mathbb{R}^n$. $W(t)$ is a $d$-dimensional Brownian motion defined on $(\Omega,\mathscr{F},\{\mathscr{F}_t\}_{t\ge 0}, \mathbb{P})$. In this section, we consider the following neutral SDDE on $\mathbb{R}^n$:
\begin{equation}\label{brownian}
\begin{split}
\mbox{d}[X(t)-D(X(t-\tau))]=&b(X(t), X(t-\tau))\mbox{d}t+\sigma(X(t), X(t-\tau))\mbox{d}W(t), t\ge 0
\end{split}
\end{equation}
with initial data $X(t)=\xi(t)\in\mathcal{L}^p_{\mathscr{F}_0}([-\tau,0];\mathbb{R}^n)$ for $t\in[-\tau,0]$, that is, $\xi$ is an $\mathscr{F}_0$-measurable $\mathcal{C}([-\tau,0];\mathbb{R}^n)$-valued random variable with $\mathbb{E}\|\xi\|^p_\infty<\infty$ for $p\ge 2$. Here, $D:\mathbb{R}^n\rightarrow\mathbb{R}^n$, and $b:\mathbb{R}^n\times\mathbb{R}^n\rightarrow\mathbb{R}^n$, $\sigma:\mathbb{R}^n\times\mathbb{R}^n\rightarrow\mathbb{R}^{n\times d}$ are continuous functions. In order to guarantee the existence and uniqueness of solutions to \eqref{brownian}, we firstly introduce functions $V_i, i=1,2,3$ such that for any $x,y\in\mathbb{R}^n$,
\begin{equation}\label{vixy}
0\le V_i(x,y)\le L_i(1+|x|^{l_i}+|y|^{l_i}), i=1,2,3
\end{equation}
for some $L_i>0, l_i\ge 1$. Furthermore, in the sequel, for any $x,y,\bar{x},\bar{y}\in\mathbb{R}^n$, we shall assume that
\begin{enumerate}
\item[{\bf (A1)}]There exists a positive constant $K_1$ such that
\begin{equation*}
\langle x-D(y)-\bar{x}+D(\bar{y}), b(x, y)-b(\bar{x}, \bar{y})\rangle\le K_1^2|x-\bar{x}|^2+|V_1(y,\bar{y})|^2|y-\bar{y}|^2,
\end{equation*}
and
\begin{equation*}
|b(x, y)-b(x, \bar{y})|\le V_1(y,\bar{y})|y-\bar{y}|.
\end{equation*}
\item[{\bf (A2)}] There exists a positive constant $K_2$ such that
\begin{equation*}
\|\sigma(x, y)-\sigma(\bar{x},\bar{y})\|\le K_2|x-\bar{x}|+V_2(y,\bar{y})|y-\bar{y}|.
\end{equation*}
\item[{\bf (A3)}] $D(0)=0$ and $|D(y)-D(\bar{y})|\le V_3(y,\bar{y})|y-\bar{y}|$.
\end{enumerate}

\begin{rem}
There are some examples such that (A1)-(A3) hold. For example, set
\begin{equation*}
\begin{split}
D(y)=-y^3,\quad b(x,y)=x-x^3+y^3,\quad \sigma(x,y)=x+y^4
\end{split}
\end{equation*}
for any $x,y\in\mathbb{R}$. It is to easy to check that (A1)-(A3) is satisfied with $V_i(x,y)=1+|x|^2+|y|^2, i=1,3$ and $V_2(x,y)=1+|x|^3+|y|^3$ for arbitrary $x,y\in\mathbb{R}$.
\end{rem}

\begin{rem}\label{remark1}
With assumption (A3), we immediately arrive at
\begin{equation}\label{Dybound}
|D(y)|\le V_3(y,0)|y|\le L_3(1+|y|+|y|^{l_3+1}).
\end{equation}
With assumptions (A1)-(A3), we have
\begin{equation*}
\begin{split}
\langle x-D(y), b(x, y)\rangle=&\langle x-D(y)-0+D(0),b(x,y)-b(0,0)\rangle+\langle x-D(y), b(0, 0)\rangle\\
\le& K_1^2|x|^2+|V_1(y,0)|^2|y|^2+\frac{1}{2}|x-D(y)|^2+\frac{1}{2}|b(0,0)|^2\\
\le&(K_1^2+1)|x|^2+|V_1(y,0)|^2|y|^2+|V_3(y,0)|^2|y|^2+\frac{1}{2}|b(0,0)|^2,
\end{split}
\end{equation*}
and
\begin{equation*}
\begin{split}
\|\sigma(x, y)\|^2\le 2\|\sigma(x, y)-\sigma(0,0)\|^2+2\|\sigma(0,0)\|^2\le 4K_2^2|x|^2+4|V_2(y,0)|^2|y|^2+2\|\sigma(0,0)\|^2.
\end{split}
\end{equation*}
Denote $K=\max\{2(K_1^2+1), 4K_2^2, |b(0,0)|^2, 2\|\sigma(0,0)\|^2\}$, and $|V(y,0)|^2=2\max\{|V_1(y,0)|^2+|V_3(y,0)|^2,2|V_2(y,0)|^2\}$, we can rewrite the above inequalities as
\begin{equation}\label{monotone}
2\langle x-D(y), b(x, y)\rangle\vee\|\sigma(x,y)\|^2\le K(1+|x|^2)+|V(y,0)|^2|y|^2.
\end{equation}
\end{rem}

Throughout the paper, we shall assume that $C$ is a positive constant, which may change line by line.

\begin{lem}\label{exactxtp}
{\rm Let (A1)-(A3) hold. Then there exists a unique global solution to \eqref{brownian}, moreover, the solution has the properties that for any $p\ge 2$, $T>0$,
\begin{equation}\label{xtp}
\mathbb{E}\left(\sup\limits_{0\le t\le T}|X(t)|^p\right)\le C,
\end{equation}
where $C=C(\xi, p, T)$ is a positive constant depending on the initial data $\xi$, $p$ and  $T$.
}
\end{lem}
\begin{proof}
With assumptions (A1)-(A3) and Remark \ref{remark1}, it is easy to see that \eqref{brownian} has a unique local solution. To verify that \eqref{brownian} admits a unique global solution, it is sufficient to show \eqref{xtp}.
Applying the It\^{o} formula and using \eqref{monotone}, we have
\begin{equation}\label{xt-dxt}
\begin{split}
&|X(t)-D(X(t-\tau))|^p=|\xi(0)-D(\xi(-\tau))|^p\\
&+p\int_0^t|X(s)-D(X(s-\tau))|^{p-2}\langle X(s)-D(X(s-\tau)), b(X(s),X(s-\tau))\rangle\mbox{d}s\\
&+\frac{p(p-1)}{2}\int_0^t|X(s)-D(X(s-\tau))|^{p-2}\|\sigma(X(s),X(s-\tau))\|^2\mbox{d}s\\
&+p\int_0^t|X(s)-D(X(s-\tau))|^{p-2}\langle X(s)-D(X(s-\tau)), \sigma(X(s),X(s-\tau))\mbox{d}W(s)\rangle\\
\le&|\xi(0)-D(\xi(-\tau))|^p+\frac{p^2K}{2}\int_0^t|X(s)-D(X(s-\tau))|^{p-2}(1+|X(s)|^2)\mbox{d}s\\
&+\frac{p^2}{2}\int_0^t|X(s)-D(X(s-\tau))|^{p-2}|V(X(s-\tau),0)|^2|X(s-\tau)|^2\mbox{d}s\\
&+p\int_0^t|X(s)-D(X(s-\tau))|^{p-2}\langle X(s)-D(X(s-\tau)), \sigma(X(s),X(s-\tau))\mbox{d}W(s)\rangle\\
=:&|\xi(0)-D(\xi(-\tau))|^p+I_1(t)+I_2(t)+I_3(t).
\end{split}
\end{equation}
Application of the Burkholder-Davis-Gundy(BDG) inequality,  the Young inequality and \eqref{monotone} yields
\begin{equation}\label{sigmadws}
\begin{split}
&\mathbb{E}\left(\sup\limits_{0\le u\le t}|I_3(u)|\right)\le C\mathbb{E}\left(\int_0^t|X(s)-D(X(s-\tau))|^{2p-2}\|\sigma(X(s),X(s-\tau))\|^2\mbox{d}s\right)^{\frac{1}{2}}\\
\le&C\mathbb{E}\left(\sup\limits_{0\le u\le t}|X(u)-D(X(u-\tau))|^{2p-2}\int_0^t\|\sigma(X(s),X(s-\tau))\|^2\mbox{d}s\right)^{\frac{1}{2}}\\
\le&\frac{1}{4}\mathbb{E}\left(\sup\limits_{0\le u\le t}|X(u)-D(X(u-\tau))|^{p}\right)+C\mathbb{E}\left(\int_0^t\|\sigma(X(s),X(s-\tau))\|^2\mbox{d}s\right)^{\frac{p}{2}}\\
\le&\frac{1}{4}\mathbb{E}\left(\sup\limits_{0\le u\le t}|X(u)-D(X(u-\tau))|^{p}\right)+C\mathbb{E}\int_0^t[1+|X(s)|^p+|V(X(s-\tau),0)|^{p}|X(s-\tau)|^p]\mbox{d}s.
\end{split}
\end{equation}
Substituting \eqref{sigmadws} into \eqref{xt-dxt}, we obtain
\begin{equation*}
\begin{split}
&\mathbb{E}\left(\sup\limits_{0\le u\le t}|X(u)-D(X(u-\tau))|^{p}\right)\le C+C\mathbb{E}\int_0^t|X(s)|^p\mbox{d}s\\
&+C\mathbb{E}\int_0^t|V(X(s-\tau),0)|^{p}|X(s-\tau)|^p\mbox{d}s
+C\mathbb{E}\int_0^t|V_3(X(s-\tau),0)|^{p}|X(s-\tau)|^p\mbox{d}s.
\end{split}
\end{equation*}
By \eqref{vixy}, we see that
\begin{equation}\label{xt-dxtp}
\begin{split}
\mathbb{E}\left(\sup\limits_{0\le u\le t}|X(u)-D(X(u-\tau))|^{p}\right)\le& C+C\mathbb{E}\int_0^t|X(s)|^p\mbox{d}s\\
&+C\mathbb{E}\int_0^t|X(s-\tau)|^{(l+1)p}\mbox{d}s,
\end{split}
\end{equation}
where $l=l_1\vee l_2\vee l_3$. Then, with \eqref{Dybound}, we derive from \eqref{xt-dxtp} that
\begin{equation*}
\begin{split}
\mathbb{E}\left(\sup\limits_{0\le u\le t}|X(u)|^{p}\right)\le&C\mathbb{E}\left(\sup\limits_{0\le u\le t}|D(X(u-\tau))|^{p}\right)+ C\mathbb{E}\left(\sup\limits_{0\le u\le t}|X(u)-D(X(u-\tau))|^{p}\right)\\
\le&C+C\mathbb{E}\left(\sup\limits_{-\tau \le u\le t-\tau}|X(u)|^{(l_3+1)p}\right)\\
&+C\mathbb{E}\int_0^t|X(s)|^p\mbox{d}s+C\mathbb{E}\int_0^t|X(s-\tau)|^{(l+1)p}\mbox{d}s\\
\le&C(\|\xi\|^{(l+1)p}_\infty, p)+C\mathbb{E}\left(\sup\limits_{-\tau \le u\le t-\tau}|X(u)|^{(l+1)p}\right)\\
&+C\int_0^t\mathbb{E}\left(\sup\limits_{0\le u\le s}|X(u)|^{p}\right)\mbox{d}s,
\end{split}
\end{equation*}
where in the last step we have used the Young inequality. The Gronwall inequality then leads to
\begin{equation*}
\begin{split}
\mathbb{E}\left(\sup\limits_{0\le u\le t}|X(u)|^{p}\right)
\le&C+C\mathbb{E}\left(\sup\limits_{0\le u\le (t-\tau)\vee 0}|X(u)|^{(l+1)p}\right).
\end{split}
\end{equation*}
For $t\in[0,\tau]$, the above inequality implies
\begin{equation*}
\begin{split}
\mathbb{E}\left(\sup\limits_{0\le t\le \tau}|X(t)|^{p}\right)\le C,
\end{split}
\end{equation*}
this further gives
\begin{equation*}
\begin{split}
\mathbb{E}\left(\sup\limits_{0\le t\le 2\tau}|X(t)|^{p}\right)\le C+C\mathbb{E}\left(\sup\limits_{0\le t\le \tau}|X(t)|^{(l+1)p}\right)\le C.
\end{split}
\end{equation*}
Finally, the desired result can be obtained with induction.
\end{proof}

We now introduce $\theta$-EM scheme for \eqref{brownian}. Given any time $T>\tau>0$, without loss of generality, assume that $T$ and $\tau$ are rational numbers, and there exist two positive integers such that $\Delta=\frac{\tau}{m}=\frac{T}{M}$, where $\Delta\in (0,1)$ is the step size. For $k=-m, \cdots, 0$, set $y_{t_k}=\xi(k\Delta)$, for $k=0, 1, \cdots,M-1$, we form
\begin{equation}\label{discrete}
\begin{split}
y_{t_{k+1}}-D(y_{t_{k+1-m}})=&y_{t_k}-D(y_{t_{k-m}})+\theta b(y_{t_{k+1}}, y_{t_{k+1-m}})\Delta\\
&+(1-\theta) b(y_{t_{k}}, y_{t_{k-m}})\Delta+\sigma(y_{t_{k}}, y_{t_{k-m}})\Delta W_{t_k},
\end{split}
\end{equation}
where $t_k=k\Delta$, $\Delta W_{t_k}=W(t_{k+1})-W(t_k)$. Here $\theta\in [0,1]$ is an additional parameter that allows us to control the implicitness of the numerical scheme. For $\theta=0$, the $\theta$-EM scheme reduces to the EM scheme, and for $\theta=1$, it is exactly the backward EM scheme. For given $y_{t_k}$, in order to guarantee a unique solution $y_{t_{k+1}}$ to \eqref{discrete}, the step size is required to satisfy $\Delta<\frac{1}{4K_1^2\theta}$ according to a fixed point theorem (see Mao and Szpruch \cite{ms132} for more information), where $K_1$ is defined as in assumption (A1). In order for simplicity, we introduce the corresponding split-step theta scheme to \eqref{brownian} as follows: For $k=-m, \cdots, -1$, set $z_{t_k}=y_{t_k}=\xi(k\Delta)$, and for $k=0,\cdots,M-1$,
\begin{equation}\label{discrete2}
\begin{cases}
y_{t_{k}}=D(y_{t_{k-m}})+z_{t_{k}}-D(z_{t_{k-m}})+\theta b(y_{t_{k}},y_{t_{k-m}})\Delta,\\
z_{t_{k+1}}=D(z_{t_{k+1-m}})+z_{t_k}-D(z_{t_{k-m}})+b(y_{t_k},y_{t_{k-m}})\Delta+\sigma(y_{t_k},y_{t_{k-m}})\Delta W_{t_k}.
\end{cases}
\end{equation}
Through computation, we can easily deduce that $y_{t_{k+1}}$ in \eqref{discrete2} can be rewritten as the form of \eqref{discrete}. Due to the implicitness of $\theta$-EM scheme, we also require $\Delta<\frac{1}{2K\theta}$, where $K$ is defined as in Remark \ref{remark1}. Thus, throughout this paper, we set $\Delta^*\in(0,(2K\vee 4K_1^2)^{-1}\theta^{-1})$, and $0<\Delta\le\Delta^*$.

\subsection{Moment Bounds}
\begin{lem}\label{pmoment}
{\rm Let (A1)-(A3) hold. Then for $\theta\in[\frac{1}{2},1]$ there exists a positive constant $C$ independent of $\Delta$ such that for $p\ge2$,
\begin{equation*}
\begin{split}
\mathbb{E}\left(\sup\limits_{0\le k\le M}|y_{t_k}|^{p}\right)\le C.
\end{split}
\end{equation*}
}
\end{lem}
\begin{proof}
By \eqref{discrete2}, we see
\begin{equation*}
\begin{split}
|z_{t_{k+1}}-D(z_{t_{k+1-m}})|^2=&|z_{t_k}-D(z_{t_{k-m}})|^2+2\langle z_{t_k}-D(z_{t_{k-m}}),b(y_{t_k},y_{t_{k-m}})\Delta\rangle\\
&+|b(y_{t_k},y_{t_{k-m}})|^2\Delta^2+|\sigma(y_{t_k},y_{t_{k-m}})\Delta W_{t_k}|^2\\
&+2\langle z_{t_k}-D(z_{t_{k-m}})+b(y_{t_k},y_{t_{k-m}})\Delta,\sigma(y_{t_k},y_{t_{k-m}})\Delta W_{t_k}\rangle\\
=&|z_{t_k}-D(z_{t_{k-m}})|^2+2\langle y_{t_k}-D(y_{t_{k-m}}),b(y_{t_k},y_{t_{k-m}})\Delta\rangle\\
&+(1-2\theta)|b(y_{t_k},y_{t_{k-m}})|^2\Delta^2+|\sigma(y_{t_k},y_{t_{k-m}})\Delta W_{t_k}|^2\\
&+2\langle y_{t_k}-D(y_{t_{k-m}})+(1-\theta)b(y_{t_k},y_{t_{k-m}})\Delta,\sigma(y_{t_k},y_{t_{k-m}})\Delta W_{t_k}\rangle.
\end{split}
\end{equation*}
Noting that $\theta\ge\frac{1}{2}$ and substituting $b(y_{t_k},y_{t_{k-m}})\Delta=\frac{1}{\theta}[y_{t_k}-D(y_{t_{k-m}})-z_{t_k}+D(z_{t_{k-m}})]$ into the last term, and using \eqref{monotone} yields
\begin{equation*}
\begin{split}
|z_{t_{k+1}}-D(z_{t_{k+1-m}})|^2\le&|z_{t_k}-D(z_{t_{k-m}})|^2+2\Delta\langle y_{t_k}-D(y_{t_{k-m}}),b(y_{t_k},y_{t_{k-m}})\rangle\\
&+|\sigma(y_{t_k},y_{t_{k-m}})\Delta W_{t_k}|^2+\frac{2}{\theta}\langle y_{t_k}-D(y_{t_{k-m}}),\sigma(y_{t_k},y_{t_{k-m}})\Delta W_{t_k}\rangle\\
&-2\frac{1-\theta}{\theta}\langle z_{t_k}-D(z_{t_{k-m}}),\sigma(y_{t_k},y_{t_{k-m}})\Delta W_{t_k}\rangle\\
\le&|z_{t_k}-D(z_{t_{k-m}})|^2+\Delta K(1+|y_{t_k}|^2)+\Delta |V(y_{t_{k-m}},0)|^2|y_{t_{k-m}}|^2\\
&+|\sigma(y_{t_k},y_{t_{k-m}})\Delta W_{t_k}|^2+\frac{2}{\theta}\langle y_{t_k}-D(y_{t_{k-m}}),\sigma(y_{t_k},y_{t_{k-m}})\Delta W_{t_k}\rangle\\
&-2\frac{1-\theta}{\theta}\langle z_{t_k}-D(z_{t_{k-m}}),\sigma(y_{t_k},y_{t_{k-m}})\Delta W_{t_k}\rangle.
\end{split}
\end{equation*}
Summing both sides from 0 to $k$, we get
\begin{equation}\label{ztk-Dztk}
\begin{split}
|z_{t_{k+1}}-D(z_{t_{k+1-m}})|^2\le&|z_{t_0}-D(z_{t_{-m}})|^2+KT+\Delta K\sum\limits_{i=0}^k|y_{t_i}|^2+\Delta \sum\limits_{i=0}^k|V(y_{t_{i-m}},0)|^2|y_{t_{i-m}}|^2\\
&+\sum\limits_{i=0}^k|\sigma(y_{t_i},y_{t_{i-m}})\Delta W_{t_i}|^2+\frac{2}{\theta}\sum\limits_{i=0}^k\langle y_{t_i}-D(y_{t_{i-m}}),\sigma(y_{t_i},y_{t_{i-m}})\Delta W_{t_i}\rangle\\
&-2\frac{1-\theta}{\theta}\sum\limits_{i=0}^k\langle z_{t_i}-D(z_{t_{i-m}}),\sigma(y_{t_i},y_{t_{i-m}})\Delta W_{t_i}\rangle.
\end{split}
\end{equation}
Using the elementary inequality
\begin{equation}\label{elementary}
\left|\sum\limits_{i=1}^nx_i\right|^p\le n^{p-1}\sum\limits_{i=1}^n|x_i|^p, \,  p\ge 1,
\end{equation}
we then have
\begin{equation*}
\begin{split}
&|z_{t_{k+1}}-D(z_{t_{k+1-m}})|^{2p}\le6^{p-1}(|z_{t_0}-D(z_{t_{-m}})|^2+KT)^{p}+6^{p-1}K^p\Delta^p\left(\sum\limits_{i=0}^k|y_{t_i}|^2\right)^{p}\\
&+6^{p-1}\Delta^p\left(\sum\limits_{i=0}^k|V(y_{t_{i-m}},0)|^2|y_{t_{i-m}}|^2\right)^p+6^{p-1}\left(\sum\limits_{i=0}^k|\sigma(y_{t_i},y_{t_{i-m}})\Delta W_{t_i}|^2\right)^p\\
&+6^{p-1}4^p\left|\sum\limits_{i=0}^k\langle y_{t_i}-D(y_{t_{i-m}}),\sigma(y_{t_i},y_{t_{i-m}})\Delta W_{t_i}\rangle\right|^p\\
&+6^{p-1}2^p\left|\sum\limits_{i=0}^k\langle z_{t_i}-D(z_{t_{i-m}}),\sigma(y_{t_i},y_{t_{i-m}})\Delta W_{t_i}\rangle\right|^p.
\end{split}
\end{equation*}
For $0<j<M$, it is easy to observe that
\begin{equation*}
\begin{split}
\mathbb{E}\left[\sup\limits_{0\le k\le j}\left(\sum\limits_{i=0}^k|y_{t_i}|^2\right)^{p}\right]\le M^{p-1}\sum\limits_{i=0}^j\mathbb{E}|y_{t_i}|^{2p},
\end{split}
\end{equation*}
and
\begin{equation*}
\begin{split}
\mathbb{E}\left[\sup\limits_{0\le k\le j}\left(\sum\limits_{i=0}^k|V(y_{t_{i-m}},0)|^2|y_{t_{i-m}}|^2\right)^p\right]\le M^{p-1}\sum\limits_{i=0}^j\mathbb{E}(|V(y_{t_{i-m}},0)|^{2p}|y_{t_{i-m}}|^{2p}).
\end{split}
\end{equation*}
By assumption (A2), we compute
\begin{equation*}
\begin{split}
&\mathbb{E}\left[\sup\limits_{0\le k\le j}\left(\sum\limits_{i=0}^k|\sigma(y_{t_i},y_{t_{i-m}})\Delta W_{t_i}|^2\right)^p\right]\\
\le& M^{p-1}\mathbb{E}\left(\sum\limits_{i=0}^j|\sigma(y_{t_i},y_{t_{i-m}})|^{2p}|\Delta W_{t_i}|^{2p}\right)
\le M^{p-1}\sum\limits_{i=0}^j\mathbb{E}|\sigma(y_{t_i},y_{t_{i-m}})|^{2p}\mathbb{E}|\Delta W_{t_i}|^{2p}\\
\le&M^{p-1}(2p-1)!!\Delta^p\sum\limits_{i=0}^j\mathbb{E}[K(1+|y_{t_i}|^2)+|V(y_{t_{i-m}},0)|^2|y_{t_{i-m}}|^2]^{p}\\
\le&C+C\sum\limits_{i=0}^j\mathbb{E}|y_{t_i}|^{2p}+C\sum\limits_{i=0}^j\mathbb{E}(|V(y_{t_{i-m}},0)|^{2p}|y_{t_{i-m}}|^{2p}).
\end{split}
\end{equation*}
With (A2)-(A3), the H\"{o}lder inequality and the BDG inequality, we get
\begin{equation*}
\begin{split}
&\mathbb{E}\left[\sup\limits_{0\le k\le j}\left|\sum\limits_{i=0}^k\langle y_{t_i}-D(y_{t_{i-m}}),\sigma(y_{t_i},y_{t_{i-m}})\Delta W_{t_i}\rangle\right|^p\right]\\
\le& C\mathbb{E}\left(\sum\limits_{i=0}^j|y_{t_i}-D(y_{t_{i-m}})|^2|\sigma(y_{t_i},y_{t_{i-m}})|^{2}\Delta\right)^{\frac{p}{2}}\\
\le& C\Delta^{\frac{p}{2}}(j+1)^{\frac{p}{2}-1}\mathbb{E}\sum\limits_{i=0}^j|y_{t_i}-D(y_{t_{i-m}})|^p[K(1+|y_{t_i}|^2)+|V(y_{t_{i-m}},0)|^2|y_{t_{i-m}}|^2]^{\frac{p}{2}}\\
\le&C+C\sum\limits_{i=0}^j\mathbb{E}|y_{t_i}|^{2p}+C\sum\limits_{i=0}^j\mathbb{E}(|V_3(y_{t_{i-m}},0)|^{2p}|y_{t_{i-m}}|^{2p})+C\sum\limits_{i=0}^j\mathbb{E}(|V(y_{t_{i-m}},0)|^{2p}|y_{t_{i-m}}|^{2p}).
\end{split}
\end{equation*}
Similarly, with (A2) and the BDG inequality again
\begin{equation*}
\begin{split}
&\mathbb{E}\left[\sup\limits_{0\le k\le j}\left|\sum\limits_{i=0}^k\langle z_{t_i}-D(z_{t_{i-m}}),\sigma(y_{t_i},y_{t_{i-m}})\Delta W_{t_i}\rangle\right|^p\right]\\
\le& C\mathbb{E}\left(\sum\limits_{i=0}^j|z_{t_i}-D(z_{t_{i-m}})|^2|\sigma(y_{t_i},y_{t_{i-m}})|^{2}\Delta\right)^{\frac{p}{2}}\\
\le&C\Delta^{\frac{p}{2}}(j+1)^{\frac{p}{2}-1}\mathbb{E}\sum\limits_{i=0}^j|z_{t_i}-D(z_{t_{i-m}})|^p[K(1+|y_{t_i}|^2)+|V(y_{t_{i-m}},0)|^2|y_{t_{i-m}}|^2]^{\frac{p}{2}}\\
\le&C+C\sum\limits_{i=0}^j\mathbb{E}(|z_{t_i}-D(z_{t_{i-m}})|^{2p})+C\sum\limits_{i=0}^j\mathbb{E}|y_{t_i}|^{2p}+C\sum\limits_{i=0}^j\mathbb{E}(|V(y_{t_{i-m}},0)|^{2p}|y_{t_{i-m}}|^{2p}).
\end{split}
\end{equation*}
Sorting this inequalities together yields
\begin{equation*}
\begin{split}
&\mathbb{E}\left(\sup\limits_{0\le k\le j+1}|z_{t_{k}}-D(z_{t_{k-m}})|^{2p}\right)\\
\le&C+C\sum\limits_{i=0}^j\mathbb{E}(|z_{t_i}-D(z_{t_{i-m}})|^{2p})+C\sum\limits_{i=0}^j\mathbb{E}|y_{t_i}|^{2p}\\
&+C\sum\limits_{i=0}^j\mathbb{E}(|V(y_{t_{i-m}},0)|^{2p}|y_{t_{i-m}}|^{2p})+C\sum\limits_{i=0}^j\mathbb{E}(|V_3(y_{t_{i-m}},0)|^{2p}|y_{t_{i-m}}|^{2p})\\
\le&C+C\sum\limits_{i=0}^j\mathbb{E}\left(\sup\limits_{0\le k\le i}|z_{t_k}-D(z_{t_{k-m}})|^{2p}\right)+C\sum\limits_{i=0}^j\mathbb{E}|y_{t_i}|^{2p}+C\sum\limits_{i=0}^j\mathbb{E}|y_{t_{i-m}}|^{2p(l+1)}.
\end{split}
\end{equation*}
The discrete Gronwall inequality then leads to
\begin{equation}\label{ztk1-Dztk1}
\begin{split}
\mathbb{E}\left(\sup\limits_{0\le k\le j+1}|z_{t_{k}}-D(z_{t_{k-m}})|^{2p}\right)
\le&C+C\sum\limits_{i=0}^j\mathbb{E}|y_{t_i}|^{2p}+C\sum\limits_{i=0}^j\mathbb{E}|y_{t_{i-m}}|^{2p(l+1)}.
\end{split}
\end{equation}
Since $y_{t_k}-D(y_{t_{k-m}})=z_{t_k}-D(z_{t_{k-m}})+\theta b(y_{t_k},y_{t_{k-m}})\Delta$, we deduce from (A1)-(A3) that
\begin{equation}\label{discreteztk}
\begin{split}
&|z_{t_{k}}-D(z_{t_{k-m}})|^{2}\\
=&|y_{t_k}-D(y_{t_{k-m}})|^2+\theta^2\Delta^2|b(y_{t_k},y_{t_{k-m}})|^2-2\theta\Delta\langle y_{t_k}-D(y_{t_{k-m}}),b(y_{t_k},y_{t_{k-m}})\rangle\\
\ge&\frac{1}{2}|y_{t_k}|^2-|V_3(y_{t_{k-m}},0)|^2|y_{t_{k-m}}|^2-\theta\Delta[K(1+|y_{t_k}|^2)+|V(y_{t_{k-m}},0)|^2|y_{t_{k-m}}|^2]\\
=&\left(\frac{1}{2}-\theta K\Delta\right)|y_{t_k}|^2-[|V_3(y_{t_{k-m}},0)|^2+\theta\Delta|V(y_{t_{k-m}},0)|^2]|y_{t_{k-m}}|^2-\theta K\Delta,
\end{split}
\end{equation}
this implies
\begin{equation*}
\begin{split}
|y_{t_k}|^2\le&\left(\frac{1}{2}-\theta K\Delta\right)^{-1}\left[|z_{t_{k}}-D(z_{t_{k-m}})|^{2}+[|V_3(y_{t_{k-m}},0)|^2+\theta\Delta|V(y_{t_{k-m}},0)|^2]|y_{t_{k-m}}|^2+\theta K\Delta\right]\\
\le&\left(\frac{1}{2}-\theta K\Delta\right)^{-1}[|z_{t_{k}}-D(z_{t_{k-m}})|^{2}+2|V(y_{t_{k-m}},0)|^2|y_{t_{k-m}}|^2+\theta K\Delta].
\end{split}
\end{equation*}
By the elementary inequality \eqref{elementary} again, we derive from \eqref{ztk1-Dztk1} that
\begin{equation*}
\begin{split}
&\mathbb{E}\left(\sup\limits_{0\le k\le j+1}|y_{t_k}|^{2p}\right)\le\left(\frac{1}{2}-\theta\Delta K\right)^{-p}3^{p-1}\bigg[\mathbb{E}\left(\sup\limits_{0\le k\le j+1}|z_{t_{k}}-D(z_{t_{k-m}})|^{2p}\right)\\
&+2^p\mathbb{E}\left(\sup\limits_{0\le k\le j+1}|V(y_{t_{k-m}},0)|^{2p}|y_{t_{k-m}}|^{2p}\right)+(\theta\Delta K)^p\bigg]\\
\le&C+C\sum\limits_{i=0}^j\mathbb{E}|y_{t_i}|^{2p}+C\sum\limits_{i=-m}^{j-m}\mathbb{E}|y_{t_{i}}|^{2p(l+1)}
+C\mathbb{E}\left(\sup\limits_{0\le k\le j+1}|y_{t_{k-m}}|^{2p(l+1)}\right)\\
\le&C+C\sum\limits_{i=0}^j\mathbb{E}\left(\sup\limits_{0\le k\le i}|y_{t_k}|^{2p}\right)+C\mathbb{E}\left(\sup\limits_{0\le k\le (j+1-m)\vee0}|y_{t_{k}}|^{2p(l+1)}\right).
\end{split}
\end{equation*}
In case of $j\le m-1$, it is obvious that
\begin{equation*}
\begin{split}
&\mathbb{E}\left(\sup\limits_{0\le k\le m}|y_{t_k}|^{2p}\right)\le C.
\end{split}
\end{equation*}
Further, for $j\le2m-1$, it follows by the Gronwall inequality that
\begin{equation*}
\begin{split}
&\mathbb{E}\left(\sup\limits_{0\le k\le 2m}|y_{t_k}|^{2p}\right)
\le C+C\mathbb{E}\left(\sup\limits_{0\le k\le m}|y_{t_{k}}|^{2p(l+1)}\right)\le C.
\end{split}
\end{equation*}
The desired assertion follows by the method of induction.
\end{proof}

\begin{rem}
{\rm For $\theta\in[0,\frac{1}{2})$, besides assumptions (A1)-(A3), if we further assume that there exists a positive constant $\bar{K}$ such that for any $x\in\mathbb{R}^n$,
\begin{equation*}
|b(x,0)|\le \bar{K}(1+|x|),
\end{equation*}
we can also show that $p$-th moment of $\theta$-EM scheme is bounded by a positive constant independent of $\Delta$.
}
\end{rem}

\subsection{Convergence Rates}
We find it is convenient to work with a continuous form of a numerical method. Noting that the split-step $\theta$-EM scheme \eqref{discrete2} can be rewritten as
\begin{equation*}
\begin{split}
& z_{t_{k+1}}-D(z_{t_{k+1-m}})=z_{t_0}-D(z_{t_{-m}})+\sum\limits_{i=0}^kb(y_{t_i},y_{t_{i-m}})\Delta+\sum\limits_{i=0}^k\sigma(y_{t_i},y_{t_{i-m}})\Delta W_{t_i}\\
=&\xi(0)-D(\xi(-\tau))-\theta b(\xi(0),\xi(-\tau))\Delta+\sum\limits_{i=0}^kb(y_{t_i},y_{t_{i-m}})\Delta+\sum\limits_{i=0}^k\sigma(y_{t_i},y_{t_{i-m}})\Delta W_{t_i}.
\end{split}
\end{equation*}
Hence, we define the corresponding continuous-time split-step $\theta$-EM solution $Z(t)$ as follows: For any $t\in[-\tau,0)$, $Z(t)=\xi(t)$, $Z(0)=\xi(0)-\theta b(\xi(0), \xi(-\tau))\Delta$, For any $t\in[0,T]$,
\begin{equation}\label{continuous}
\mbox{d}[Z(t)-D(Z(t-\tau))]=b(\bar{Y}(t), \bar{Y}(t-\tau))\mbox{d}t+\sigma(\bar{Y}(t), \bar{Y}(t-\tau))\mbox{d}W(t),
\end{equation}
where $\bar{Y}(t)$ is defined by
\begin{equation*}
\bar{Y}(t):=y_{t_k} \quad \mbox{for} \quad t\in[t_k, t_{k+1}),
\end{equation*}
thus $\bar{Y}(t-\tau)=y_{t_{k-m}}$. We now define the continuous $\theta$-EM solution $Y(t)$ as follows:
\begin{equation}\label{ytztrelation}
Y(t)-D(Y(t-\tau))=Z(t)-D(Z(t-\tau))+\theta b(Y(t),Y(t-\tau))\Delta.
\end{equation}
It can be verified that $Y(t_k)=y_{t_k}$, $k=-m,\cdots,M$. In order to obtain convergence rate, we impose another assumption as follows:
\begin{enumerate}
\item[{\bf (A4)}]For $x,\bar{x},y\in\mathbb{R}^n$,$|b(x, y)-b(\bar{x}, y)|\le V_1(x,\bar{x})|x-\bar{x}|$.
\end{enumerate}

\begin{rem}
From assumptions (A1) and (A4), one sees that
\begin{equation*}
|b(x,y)|\le |b(x,y)-b(x,0)|+|b(x,0)-b(0,0)|+|b(0,0)|\le V_1(x,0)|x|+V_1(y,0)|y|+|b(0,0)|,
\end{equation*}
and further,
\begin{equation*}
|b(x,y)-b(\bar{x},\bar{y})|\le |b(x,y)-b(\bar{x},y)|+|b(\bar{x},y)-b(\bar{x},\bar{y})|\le V_1(x,\bar{x})|x-\bar{x}|+V_1(y,\bar{y})|y-\bar{y}|.
\end{equation*}
\end{rem}

\begin{lem}\label{ytytk}
{\rm
Consider the $\theta$-EM scheme \eqref{discrete}, and let (A1)-(A4) hold. Then, for any $p\ge 2$, the continuous form of $\theta$-EM scheme solution $Y(t)$ has the following properties,
\begin{equation*}
\mathbb{E}\left(\sup\limits_{0\le t\le T}|Y(t)|^p\right)\le C,
\end{equation*}
and
\begin{equation*}
\mathbb{E}\left(\sup\limits_{0\le t\le T}|Y(t)-\bar{Y}(t)|^p\right)\le C\Delta^{\frac{p}{2}},
\end{equation*}
where $C$ is a constant independent of $\Delta$.
}
\end{lem}

\begin{proof}
For any $p\ge 2$, by the elementary inequality \eqref{elementary}, we have
\begin{equation*}
\begin{split}
&\mathbb{E}\left(\sup\limits_{0\le u\le t}|Z(u)-D(Z(u-\tau))|^p\right)\\
\le &3^{p-1}|Z(0)-D(Z(-\tau))|^p+3^{p-1}\mathbb{E}\left(\sup\limits_{0\le u\le t}\left|\int_0^u b(\bar{Y}(s), \bar{Y}(s-\tau))\mbox{d}s\right|^p\right)\\
&+3^{p-1}\mathbb{E}\left(\sup\limits_{0\le u\le t}\left|\int_{0}^u\sigma(\bar{Y}(s), \bar{Y}(s-\tau))\mbox{d}W(s)\right|^p\right).
\end{split}
\end{equation*}
Using the H\"{o}lder inequality, the BDG inequality, and together with (A2)-(A4), Lemma \ref{pmoment} yields
\begin{equation}\label{supcontin}
\begin{split}
&\mathbb{E}\left(\sup\limits_{0\le u\le t}|Z(u)-D(Z(u-\tau))|^p\right)\\
\le &3^{p-1}|Z(0)-D(Z(-\tau))|^p+3^{p-1}t^{p-1}\mathbb{E}\int_0^t\left|b(\bar{Y}(s), \bar{Y}(s-\tau))\right|^p\mbox{d}s\\
&+C\mathbb{E}\left(\int_{0}^t\|\sigma(\bar{Y}(s), \bar{Y}(s-\tau))\|^2\mbox{d}s\right)^{\frac{p}{2}}\\
\le& C+C\mathbb{E}\int_0^t\left[|V_1(\bar{Y}(s),0)|^p|\bar{Y}(s)|^p+|V_1(\bar{Y}(s-\tau),0)|^p|\bar{Y}(s-\tau)|^p+|b(0,0)|^p\right]\mbox{d}s\\
&+C\mathbb{E}\int_{0}^t[|\bar{Y}(s)|^p+|V(\bar{Y}(s-\tau))|^p|\bar{Y}(s-\tau)|^p]\mbox{d}s\\
\le&C+C\mathbb{E}\int_0^t|\bar{Y}(s)|^p\mbox{d}s+C\mathbb{E}\int_0^t|\bar{Y}(s)|^{(l+1)p}\mbox{d}s\le C.
\end{split}
\end{equation}
With the relationship \eqref{ytztrelation}, similar to \eqref{discreteztk}, we get
\begin{equation*}
\begin{split}
|Y(t)|^2\le&\left(\frac{1}{2}-\theta K\Delta\right)^{-1}[|Z(t)-D(Z(t-\tau))|^{2}+2|V(Y(t-\tau),0)|^2|Y(t-\tau)|^2+\theta K\Delta].
\end{split}
\end{equation*}
We then derive from \eqref{supcontin} that
\begin{equation*}
\begin{split}
\mathbb{E}\left(\sup\limits_{0\le u\le t}|Y(u)|^p\right)\le& C+C\mathbb{E}\left(\sup\limits_{0\le u\le t}|Z(u)-D(Z(u-\tau))|^p\right)+C\mathbb{E}\left(\sup\limits_{0\le u\le t}|Y(u-\tau)|^{(l+1)p}\right)\\
\le&C+C\mathbb{E}\left(\sup\limits_{0\le u\le (t-\tau)\vee 0}|Y(u)|^{(l+1)p}\right).
\end{split}
\end{equation*}
Following the procedure of Lemma \ref{exactxtp}, we can show that the $p$-th moment of $Y(t)$ is bounded by a positive constant $C$. Denote by $\bar{Z}(t):=z_{t_k}$ for $t\in[t_k, t_{k+1})$, we see from \eqref{continuous} that
\begin{equation*}
Z(t)-D(Z(t-\tau))-\bar{Z}(t)+D(\bar{Z}(t-\tau))=\int_{t_k}^tb(\bar{Y}(s), \bar{Y}(s-\tau))\mbox{d}s+\int_{t_k}^t\sigma(\bar{Y}(s), \bar{Y}(s-\tau))\mbox{d}W(s),
\end{equation*}
Denote by $\Phi(Z(t),\bar{Z}(t)):=Z(t)-D(Z(t-\tau))-\bar{Z}(t)+D(\bar{Z}(t-\tau))$, then
\begin{equation*}
\begin{split}
&\mathbb{E}\left(\sup\limits_{t_k\le t<t_{k+1}}|\Phi(Z(t),\bar{Z}(t))|^p\right)\le2^{p-1}\mathbb{E}\left(\sup\limits_{t_k\le t<t_{k+1}}\left|\int_{t_k}^t b(\bar{Y}(s), \bar{Y}(s-\tau))\mbox{d}s\right|^p\right)\\
&+2^{p-1}\mathbb{E}\left(\sup\limits_{t_k\le t<t_{k+1}}\left|\int_{t_k}^t\sigma(\bar{Y}(s), \bar{Y}(s-\tau))\mbox{d}W(s)\right|^p\right).
\end{split}
\end{equation*}
With (A2), (A4), Lemma \ref{pmoment}, the H\"{o}lder inequality, and the BDG inequality, we get
\begin{equation}\label{phizt}
\begin{split}
\mathbb{E}\left(\sup\limits_{t_k\le t<t_{k+1}}|\Phi(Z(t),\bar{Z}(t))|^p\right)
\le& 2^{p-1}\Delta^{p-1}\mathbb{E}\left[\int_{t_k}^{t_{k+1}}\left|b(\bar{Y}(s), \bar{Y}(s-\tau))\right|^p\mbox{d}s\right]\\
&+C\mathbb{E}\left[\int_{t_k}^{t_{k+1}}\left\|\sigma(\bar{Y}(s), \bar{Y}(s-\tau))\right\|^2\mbox{d}s\right]^{\frac{p}{2}}\\
\le &C\Delta^{p}+C\Delta^{\frac{p}{2}}\le C\Delta^{\frac{p}{2}}.
\end{split}
\end{equation}
On the other hand, we have the following relationship between $\bar{Y}(t)$ and $\bar{Z}(t)$,
\begin{equation}\label{barytzrelation}
\bar{Y}(t)-D(\bar{Y}(t-\tau))=\bar{Z}(t)-D(\bar{Z}(t-\tau))+\theta b(\bar{Y}(t),\bar{Y}(t-\tau))\Delta.
\end{equation}
Combing \eqref{ytztrelation} and \eqref{barytzrelation} gives
\begin{equation*}
\begin{split}
Y(t)-\bar{Y}(t)=&D(Y(t-\tau))-D(\bar{Y}(t-\tau))+\Phi(Z(t),\bar{Z}(t))\\
&+\theta[b(Y(t),Y(t-\tau))-b(\bar{Y}(t),\bar{Y}(t-\tau))]\Delta.
\end{split}
\end{equation*}
Using similar skills of \eqref{discreteztk}, we derive from (A1) and (A3)
\begin{equation*}
\begin{split}
|Y(t)-\bar{Y}(t)|^2\le&\left(\frac{1}{2}-2\theta K_1^2\Delta\right)^{-1}\bigg\{|\Phi(Z(t),\bar{Z}(t))|^{2}+[|V_3(Y(t-\tau),\bar{Y}(t-\tau))|^2 \\ &+2\theta\Delta|V(Y(t-\tau),\bar{Y}(t-\tau))|^2]|Y(t-\tau)-\bar{Y}(t-\tau)|^2+2\theta\Delta K_1^2\bigg\}.
\end{split}
\end{equation*}
Obviously, due to \eqref{phizt},
\begin{equation*}
\begin{split}
&\mathbb{E}\left(\sup\limits_{0\le u\le t}|Y(u)-\bar{Y}(u)|^p\right)\\
&\le C\mathbb{E}\left(\sup\limits_{0\le u\le t}|\Phi(Z(u),\bar{Z}(u))|^{p}\right)
+C\mathbb{E}\left(\sup\limits_{0\le u\le t}|Y(u-\tau)-\bar{Y}(u-\tau)|^p\right)\\
&\le C\Delta^{\frac{p}{2}}+C\mathbb{E}\left(\sup\limits_{0\le u\le (t-\tau)\vee0}|Y(u)-\bar{Y}(u)|^p\right).
\end{split}
\end{equation*}
The desired result follows by repeating the techniques of Lemma \ref{exactxtp}.
\end{proof}

\begin{thm}\label{theorem1}
{\rm Let assumptions (A1)-(A4) hold and $\theta\in\left[\frac{1}{2},1\right]$. Then it holds that the $\theta$-EM solution $Y(t)$ converges to the exact solution $X(t)$ in $\mathcal{L}^p$ sense with order $\frac{1}{2}$, i.e.,
\begin{equation*}
\mathbb{E}\left(\sup\limits_{0\le t\le T}|Y(t)-X(t)|^p\right)\le C\Delta^{\frac{p}{2}}
\end{equation*}
for $p\ge 2$.
}
\end{thm}
\begin{proof}
Denote by $e(t):=Z(t)-D(Z(t-\tau))-X(t)+D(X(t-\tau))$, then
\begin{equation*}
\begin{split}
e(t)=&e(0)+\int_0^t[b(\bar{Y}(s),\bar{Y}(s-\tau))-b(X(s),X(s-\tau))]\mbox{d}s\\
&+\int_0^t[\sigma(\bar{Y}(s),\bar{Y}(s-\tau))-\sigma(X(s),X(s-\tau))]\mbox{d}W(s),
\end{split}
\end{equation*}
where $e(0)=-\theta b(\xi(0), \xi(-\tau))\Delta$. Application of the It\^{o} formula yields
\begin{equation*}
\begin{split}
|e(t)|^p=&|e(0)|^p+p\int_0^t |e(s)|^{p-2}\langle e(s),b(\bar{Y}(s),\bar{Y}(s-\tau))-b(X(s),X(s-\tau))\rangle\mbox{d}s\\
&+\frac{1}{2}p(p-1)\int_0^t |e(s)|^{p-2}\|\sigma(\bar{Y}(s),\bar{Y}(s-\tau))-\sigma(X(s),X(s-\tau))\|^2\mbox{d}s\\
&+p\int_0^t |e(s)|^{p-2}\langle e(s),\sigma(\bar{Y}(s),\bar{Y}(s-\tau))-\sigma(X(s),X(s-\tau))\mbox{d}W(s)\rangle.
\end{split}
\end{equation*}
Rewrite $|e(t)|^p$ as
\begin{equation*}
\begin{split}
|e(t)|^p\le&|e(0)|^p+p\int_0^t |e(s)|^{p-2}\langle e(s),b(\bar{Y}(s),\bar{Y}(s-\tau))-b(Y(s),\bar{Y}(s-\tau))\rangle\mbox{d}s\\
&+p\int_0^t |e(s)|^{p-2}\langle e(s),b(Y(s),\bar{Y}(s-\tau))-b(Y(s),Y(s-\tau))\rangle\mbox{d}s\\
&+p\int_0^t |e(s)|^{p-2}\langle e(s),b(Y(s),Y(s-\tau))-b(X(s),X(s-\tau))\rangle\mbox{d}s\\
&+\frac{3}{2}p(p-1)\int_0^t |e(s)|^{p-2}\|\sigma(\bar{Y}(s),\bar{Y}(s-\tau))-\sigma(Y(s),\bar{Y}(s-\tau))\|^2\mbox{d}s\\
&+\frac{3}{2}p(p-1)\int_0^t |e(s)|^{p-2}\|\sigma(Y(s),\bar{Y}(s-\tau))-\sigma(Y(s),Y(s-\tau))\|^2\mbox{d}s\\
&+\frac{3}{2}p(p-1)\int_0^t |e(s)|^{p-2}\|\sigma(Y(s),Y(s-\tau))-\sigma(X(s),X(s-\tau))\|^2\mbox{d}s\\
&+p\int_0^t |e(s)|^{p-2}\langle e(s),\sigma(\bar{Y}(s),\bar{Y}(s-\tau))-\sigma(X(s),X(s-\tau))\mbox{d}W(s)\rangle\\
=:&|e(0)|^p+H_1(t)+H_2(t)+H_3(t)+H_4(t)+H_5(t)+H_6(t)+H_7(t).
\end{split}
\end{equation*}
By (A4), Lemma \ref{ytytk} and the H\"{o}lder inequality,
\begin{equation*}
\begin{split}
& \mathbb{E}\left(\sup\limits_{0\le u\le t}|H_1(u)|\right)\\
\le&C\mathbb{E}\int_0^t|e(s)|^{p}\mbox{d}s+C\mathbb{E}\int_0^t|b(\bar{Y}(s),\bar{Y}(s-\tau))-b(Y(s),\bar{Y}(s-\tau))|^p\mbox{d}s\\
\le &C\mathbb{E}\int_0^t[|Y(s)-X(s)|^p+|V_3(Y(s-\tau),X(s-\tau))|^p|Y(s-\tau)-X(s-\tau)|^p\\
&+\theta^p\Delta^p|b(Y(s),Y(s-\tau))|^p]\mbox{d}s+C\mathbb{E}\int_0^t|V_1(\bar{Y}(s),Y(s))|^p|\bar{Y}(s)-Y(s)|^p\mbox{d}s\\
\le &C\mathbb{E}\int_0^t|Y(s)-X(s)|^p\mbox{d}s+C\int_0^t[\mathbb{E}|V_3(Y(s-\tau),X(s-\tau))|^{2p}]^{\frac{1}{2}}[\mathbb{E}|Y(s-\tau)-X(s-\tau)|^{2p}]^{\frac{1}{2}}\mbox{d}s\\
&+C\Delta^p\mathbb{E}\int_0^t[V_1(Y(s),0)|Y(s)|+V_1(Y(s-\tau),0)|Y(s-\tau)|+|b(0,0)|]^p\mbox{d}s\\
&+C\int_0^t[\mathbb{E}|V_1(\bar{Y}(s),Y(s))|^{2p}]^{\frac{1}{2}}[\mathbb{E}|\bar{Y}(s)-Y(s)|^{2p}]^{\frac{1}{2}}\mbox{d}s\\
\le&C\int_0^t\mathbb{E}\left(\sup\limits_{0\le u\le s}|Y(u)-X(u)|^p\right)\mbox{d}s+C\Delta^{p}+C\Delta^{\frac{p}{2}}.
\end{split}
\end{equation*}
By (A1), Lemma \ref{ytytk} and the H\"{o}lder inequality,
\begin{equation*}
\begin{split}
& \mathbb{E}\left(\sup\limits_{0\le u\le t}|H_2(u)|\right)\\
\le&C\mathbb{E}\int_0^t|e(s)|^{p}\mbox{d}s+C\mathbb{E}\int_0^t|b(Y(s),\bar{Y}(s-\tau))-b(Y(s),Y(s-\tau))|^p\mbox{d}s\\
\le &C\mathbb{E}\int_0^t|e(s)|^{p}\mbox{d}s+C\mathbb{E}\int_0^t|V(\bar{Y}(s-\tau),Y(s-\tau))|^p|\bar{Y}(s-\tau)-Y(s-\tau)|^p\mbox{d}s\\
\le&C\int_0^t\mathbb{E}\left(\sup\limits_{0\le u\le s}|Y(u)-X(u)|^p\right)\mbox{d}s+C\Delta^{p}+C\Delta^{\frac{p}{2}}.
\end{split}
\end{equation*}
Due to (A1)-(A2), Lemma \ref{ytytk} and the H\"{o}lder inequality,
\begin{equation*}
\begin{split}
& \mathbb{E}\left(\sup\limits_{0\le u\le t}|H_3(u)+H_6(u)|\right)\le C\mathbb{E}\int_0^t|e(s)|^{p-2}|Y(s)-X(s)|^{2}\mbox{d}s\\
&+C\mathbb{E}\int_0^t|e(s)|^{p-2}|V_1(Y(s-\tau),X(s-\tau))|^2|Y(s-\tau)-X(s-\tau)|^{2}\mbox{d}s\\
&+C\mathbb{E}\int_0^t|e(s)|^{p-2}|V_2(Y(s-\tau),X(s-\tau))|^2|Y(s-\tau)-X(s-\tau)|^{2}\mbox{d}s\\
&+C\mathbb{E}\int_0^t |e(s)|^{p-2}|\theta b(Y(s),Y(s-\tau))\Delta||b(Y(s),Y(s-\tau))-b(X(s),X(s-\tau))|\mbox{d}s\\
\le&C\mathbb{E}\int_0^t|Y(s)-X(s)|^{p}\mbox{d}s+C\Delta^p\le C\int_0^t \mathbb{E}\left(\sup\limits_{0\le u\le s}|Y(u)-X(u)|^p\right)\mbox{d}s+C\Delta^{p}.
\end{split}
\end{equation*}
In the same way to estimate $H_1(t)$ and $H_2(t)$, we get
\begin{equation*}
\begin{split}
\mathbb{E}\left(\sup\limits_{0\le u\le t}|H_4(u)|\right)\le&C\int_0^t\mathbb{E}\left(\sup\limits_{0\le u\le s}|Y(u)-X(u)|^p\right)\mbox{d}s+C\Delta^{p}+C\Delta^{\frac{p}{2}},
\end{split}
\end{equation*}
and
\begin{equation*}
\begin{split}
 \mathbb{E}\left(\sup\limits_{0\le u\le t}|H_5(u)|\right)
\le C\int_0^t \mathbb{E}\left(\sup\limits_{0\le u\le s}|Y(u)-X(u)|^p\right)\mbox{d}s+C\Delta^{p}+C\Delta^{\frac{p}{2}}.
\end{split}
\end{equation*}
Furthermore, by (A3), Lemma \ref{ytytk}, the BDG inequality and the H\"{o}lder inequality, we compute
\begin{equation*}
\begin{split}
& \mathbb{E}\left(\sup\limits_{0\le u\le t}|H_7(u)|\right)\le C\mathbb{E}\left(\int_0^t |e(s)|^{2p-2}\|\sigma(\bar{Y}(s),\bar{Y}(s-\tau))-\sigma(X(s),X(s-\tau))\|^{2}\mbox{d}s\right)^{\frac{1}{2}}\\
\le&\frac{1}{4}\mathbb{E}\left(\sup\limits_{0\le u\le t}|e(u)|^p\right)+C\int_0^t \mathbb{E}\left(\sup\limits_{0\le u\le s}|Y(u)-X(u)|^p\right)\mbox{d}s+C\Delta^{\frac{p}{2}}+C\Delta^{p}.
\end{split}
\end{equation*}
Consequently, by sorting $H_1(t)-H_7(t)$ together, we arrive at
\begin{equation*}
\begin{split}
\mathbb{E}\left(\sup\limits_{0\le u\le t}|e(u)|^p\right)\le C\int_0^t \mathbb{E}\left(\sup\limits_{0\le u\le s}|Y(u)-X(u)|^p\right)\mbox{d}s+C\Delta^{\frac{p}{2}}.
\end{split}
\end{equation*}
By the definition of $e(t)$, we derive from (A3) that
\begin{equation*}
\begin{split}
|Y(t)-X(t)|^p\le& 3^{p-1}|e(t)|^p+3^{p-1}|\theta b(Y(t),Y(t-\tau))\Delta|^p+3^{p-1}|D(Y(t-\tau))-D(X(t-\tau))|^p\\
\le&3^{p-1}|e(t)|^p+3^{p-1}\theta^p\Delta^p|b(Y(t),Y(t-\tau))|^p\\
&+3^{p-1}|V_3(Y(t-\tau),X(t-\tau))|^p|Y(t-\tau)-X(t-\tau)|^p.
\end{split}
\end{equation*}
Taking (A1) and Lemma \ref{ytytk} into consideration,
\begin{equation*}
\begin{split}
\mathbb{E}\left(\sup\limits_{0\le u\le t}|Y(u)-X(u)|^p\right)
\le&C\mathbb{E}\left(\sup\limits_{0\le u\le t}|e(u)|^p\right)+C\Delta^{p}+C\mathbb{E}\left(\sup\limits_{0\le u\le t}|Y(u-\tau)-X(u-\tau)|^p\right)\\
\le&C\int_0^t\mathbb{E}\left(\sup\limits_{0\le u\le s}|Y(u)-X(u)|^p\right)\mbox{d}s+C\Delta^{\frac{p}{2}}\\
&+C\Delta^{p}+C\mathbb{E}\left(\sup\limits_{0\le u\le (t-\tau)\vee0}|Y(u)-X(u)|^p\right).
\end{split}
\end{equation*}
The Gronwall inequality yields
\begin{equation*}
\begin{split}
\mathbb{E}\left(\sup\limits_{0\le u\le t}|Y(u)-X(u)|^p\right)
\le&C\Delta^{\frac{p}{2}}+C\mathbb{E}\left(\sup\limits_{0\le u\le (t-\tau)\vee0}|Y(u)-X(u)|^p\right).
\end{split}
\end{equation*}
Again, the desired result follows by the induction.

\end{proof}

With strong convergence rate given in Theorem \ref{theorem1}, we can easily show the following result on almost sure convergence.
\begin{thm}\label{astheorem}
{\rm Let the conditions of Theorem \ref{theorem1} hold. Then the continuous form of $\theta$-EM scheme \eqref{discrete} converges to the exact solution of \eqref{brownian} almost surely with order $\alpha<\frac{1}{2}$, i.e., there exists a finite random variable $\zeta_\alpha$ such that
\begin{equation*}
\sup\limits_{0\le t\le T}|Y(t)-X(t)|\le\zeta_\alpha\Delta^\alpha
\end{equation*}
for $\alpha\in(0,\frac{1}{2})$.
}
\end{thm}
\begin{proof}
Define a sequence $\Delta_k, k=1,2,\cdots$ such that $\Delta=\Delta_1>\Delta_2>\cdots$ and $\sum\limits_k\Delta_k^{\left(\frac{1}{2}-\alpha\right)p}<\infty, p\ge 2$. By the Chebyshev inequality and Theorem \ref{theorem1}, for $\alpha<\frac{1}{2}$
\begin{equation*}
\begin{split}
\sum\limits_k\mathbb{P}\left(\sup\limits_{0\le t\le T}|Y(t)-X(t)|>\Delta_k^\alpha\right)\le&\sum\limits_k\mathbb{E}\left(\sup\limits_{0\le t\le T}|Y(t)-X(t)|^p\right)\Delta_k^{-\alpha p}\\
\le& C\sum\limits_k\Delta_k^{\left(\frac{1}{2}-\alpha\right)p}<\infty.
\end{split}
\end{equation*}
The Borel-Cantelli lemma implies that there exists a finite random variable $\zeta_\alpha$ such that
\begin{equation*}
\sup\limits_{0\le t\le T}|Y(t)-X(t)|\le\zeta_\alpha\Delta^\alpha.
\end{equation*}
\end{proof}

\section{Convergence Rates for Pure Jumps Case}
In this section, we further introduce some notation. Let $N(\cdot,\cdot)$ be a Poisson random process with characteristic measure $\lambda$ on a measurable subset $U$ of $[0,\infty)$ such that $\lambda(U)<\infty$, then $\tilde{N}(\mbox{d}u,\mbox{d}t)=N(\mbox{d}u,\mbox{d}t)-\lambda(\mbox{d}u)\mbox{d}t$ is a compensated martingale process. We consider the following neutral SDDE with jumps on $\mathbb{R}^n$:
\begin{equation}\label{jump}
\begin{split}
\mbox{d}[X(t)-D(X(t-\tau))]=&b(X(t), X(t-\tau))\mbox{d}t+\int_U h(X(t), X(t-\tau),u)\tilde{N}(\mbox{d}u,\mbox{d}t), t\ge 0
\end{split}
\end{equation}
with initial data $X(\theta)=\xi(\theta)\in\mathcal{L}^p_{\mathscr{F}_0}([-\tau,0];\mathbb{R}^n)$ for $\theta\in[-\tau,0]$, i.e., $\xi$ is an $\mathscr{F}_0$-measurable $\mathcal{D}([-\tau,0];\mathbb{R}^n)$-valued random variable such that $\mathbb{E}\|\xi\|^p_\infty<\infty$ for $p\ge 2$, where $\mathcal{D}([-\tau,0];\mathbb{R}^n)$ denotes the space of all c\'{a}dl\'{a}g paths $\zeta:[-\tau, 0]\rightarrow\mathbb{R}^n$ with uniform norm $\|\zeta\|_\infty:=\sup_{-\tau\le\theta\le 0}|\zeta (\theta)|$. Here, $D:\mathbb{R}^n\rightarrow\mathbb{R}^n$, and $b:\mathbb{R}^n\times\mathbb{R}^n\rightarrow\mathbb{R}^n$, $h:\mathbb{R}^n\times\mathbb{R}^n\times U\rightarrow\mathbb{R}^n$ are measurable functions. We further assume that $b$ is a continuous function and $\int_U|u|^p\lambda(\mbox{d}u)<\infty$ for $p\ge 2$. Similar to Brownian motion case, for $x,y,\bar{x},\bar{y}\in\mathbb{R}^n$, we shall assume that:
\begin{enumerate}
\item[{\bf (A5)}] There exist positive constants $\bar{K}_2$ and $r\ge 1$ such that
\begin{equation*}
|h(x, y,u)-h(\bar{x},\bar{y},u)|\le [\bar{K}_2|x-\bar{x}|+V_2(y,\bar{y})|y-\bar{y}|]|u|^r, \mbox{and} |h(0,0,u)|\le |u|^r.
\end{equation*}
\end{enumerate}
\begin{rem}\label{remarkhbound}
With assumption (A5), we have
\begin{equation*}
\begin{split}
|h(x, y,u)|\le&|h(x, y,u)-h(0, 0,u)|+|h(0, 0,u)|\le[1+\bar{K}_2|x|+V_2(y,0)|y|]|u|^r.
\end{split}
\end{equation*}
\end{rem}

\begin{lem}\label{jumpexactxtp}
{\rm Let (A1), (A3) and (A5) hold. Then there exists a unique global solution to \eqref{jump}, moreover, the solution has the property that for any $p\ge 2$, $T>0$,
\begin{equation}\label{jumpxtp}
\mathbb{E}\left(\sup\limits_{0\le t\le T}|X(t)|^p\right)\le C,
\end{equation}
where $C=C(\xi, p, T)$ is a positive constant which only depends on the initial data $\xi$ and $p, T$.
}
\end{lem}
\begin{proof}
We omit the proof here since it is similar to that of Lemma \ref{exactxtp}.
\end{proof}

We now introduce the $\theta$-EM scheme for \eqref{jump}. Given any time $T>\tau>0$, assume that $T$ and $\tau$ are rational numbers, and there exists two positive integers such that $\Delta=\frac{\tau}{m}=\frac{T}{M}$, where $\Delta\in (0,1)$ is the step size. For $k=-m, \cdots, 0$, set $y_{t_k}=\xi(k\Delta)$; For $k=0, 1, \cdots,M-1$, we form
\begin{equation}\label{jumpdiscrete}
\begin{split}
y_{t_{k+1}}-D(y_{t_{k+1-m}})=&y_{t_k}-D(y_{t_{k-m}})+\theta b(y_{t_{k+1}}, y_{t_{k+1-m}})\Delta\\
&+(1-\theta) b(y_{t_{k}}, y_{t_{k-m}})\Delta+\int_Uh(y_{t_{k}}, y_{t_{k-m}},u)\Delta\tilde{N}_k(\mbox{d}u),
\end{split}
\end{equation}
where $t_k=k\Delta$, and $\Delta \tilde{N}_k(\mbox{d}u)=\tilde{N}(t_{k+1},\mbox{d}u)-\tilde{N}(t_k,\mbox{d}u)$. Here $\theta\in [0,1]$ is an additional parameter that allows us to control the implicitness of the numerical scheme. For $\theta=0$, the $\theta$-EM scheme reduces to the EM scheme, and for $\theta=1$, it is the backward EM scheme. Here we always assume $\theta\ge 1/2$. The corresponding split-step $\theta$-EM scheme to \eqref{jump} is defined as follows: For $k=-m, \cdots, -1$, set $z_{t_k}=y_{t_k}=\xi(k\Delta)$; For $k=0, 1, \cdots,M-1$,
\begin{equation}\label{jumpdiscrete2}
\begin{cases}
y_{t_{k}}=D(y_{t_{k-m}})+z_{t_{k}}-D(z_{t_{k-m}})+\theta b(y_{t_{k}},y_{t_{k-m}})\Delta,\\
z_{t_{k+1}}=D(z_{t_{k+1-m}})+z_{t_k}-D(z_{t_{k-m}})+b(y_{t_k},y_{t_{k-m}})\Delta+\int_Uh(y_{t_{k}}, y_{t_{k-m}},u)\Delta\tilde{N}_k(\mbox{d}u).
\end{cases}
\end{equation}
It is easy to see $y_{t_{k+1}}$ in \eqref{jumpdiscrete2} can be rewritten as the form of \eqref{jumpdiscrete}. Due to the implicitness of $\theta$-EM scheme, we require $0<\Delta\le\Delta^*$, where $\Delta^*\in(0,(2K\vee 4K_1^2)^{-1}\theta^{-1})$, $K_1$ and $K$ are defined as in (A1) and Remark \ref{remark1} with $\sigma\equiv{\bf 0}$ respectively.

\subsection{Moment Bounds}
Firstly, we introduce an important lemma coming from \cite{mpr10}.
\begin{lem}\label{usefullemma}
{\rm Let $\phi:\mathbb{R}_+\times U\rightarrow\mathbb{R}^n$ be progressively measurable and assume that the right side is finite.
Then there exists a positive constant $C$ such that
\begin{equation*}
\begin{split}
&\mathbb{E}\left(\sup\limits_{0\le s\le t}\left|\int_0^s\int_U\phi(r-,u)\tilde{N}(\mbox{d}u,\mbox{d}r)\right|^p\right)\le C\mathbb{E}\int_0^t\int_U|\phi(s,u)|^p\lambda(\mbox{d}u)\mbox{d}s
\end{split}
\end{equation*}
for $p\ge 2$.
}
\end{lem}
\begin{lem}\label{jumppmoment}
{\rm Let (A1), (A3) and (A5) hold. Then, there exists a positive constant $C$ independent of $\Delta$ such that
\begin{equation*}
\begin{split}
\mathbb{E}\left(\sup\limits_{0\le k\le M}|y_{t_k}|^{p}\right)\le C
\end{split}
\end{equation*}
for $p\ge 2$.
}
\end{lem}
\begin{proof}
It is easy to see from \eqref{jumpdiscrete2}
\begin{equation*}
\begin{split}
&|z_{t_{k+1}}-D(z_{t_{k+1-m}})|^2=|z_{t_k}-D(z_{t_{k-m}})|^2+2\langle z_{t_k}-D(z_{t_{k-m}}),b(y_{t_k},y_{t_{k-m}})\Delta\rangle\\
&+|b(y_{t_k},y_{t_{k-m}})|^2\Delta^2+\left|\int_Uh(y_{t_{k}}, y_{t_{k-m}},u)\Delta\tilde{N}_k(\mbox{d}u)\right|^2\\
&+2\left\langle z_{t_k}-D(z_{t_{k-m}})+b(y_{t_k},y_{t_{k-m}})\Delta,\int_Uh(y_{t_{k}}, y_{t_{k-m}},u)\Delta\tilde{N}_k(\mbox{d}u)\right\rangle\\
=&|z_{t_k}-D(z_{t_{k-m}})|^2+2\langle y_{t_k}-D(y_{t_{k-m}}),b(y_{t_k},y_{t_{k-m}})\Delta\rangle\\
&+(1-2\theta)|b(y_{t_k},y_{t_{k-m}})|^2\Delta^2+\left|\int_Uh(y_{t_{k}}, y_{t_{k-m}},u)\Delta\tilde{N}_k(\mbox{d}u)\right|^2\\
&+2\left\langle y_{t_k}-D(y_{t_{k-m}})+(1-\theta)b(y_{t_k},y_{t_{k-m}})\Delta,\int_Uh(y_{t_{k}}, y_{t_{k-m}},u)\Delta\tilde{N}_k(\mbox{d}u)\right\rangle.
\end{split}
\end{equation*}
Applying \eqref{jumpdiscrete} to the last term and using assumption (A1) lead to
\begin{equation*}
\begin{split}
&|z_{t_{k+1}}-D(z_{t_{k+1-m}})|^2\le|z_{t_k}-D(z_{t_{k-m}})|^2+2\Delta\langle y_{t_k}-D(y_{t_{k-m}}),b(y_{t_k},y_{t_{k-m}})\rangle\\
&+\left|\int_Uh(y_{t_{k}}, y_{t_{k-m}},u)\Delta\tilde{N}_k(\mbox{d}u)\right|^2+\frac{2}{\theta}\left\langle y_{t_k}-D(y_{t_{k-m}}),\int_Uh(y_{t_{k}}, y_{t_{k-m}},u)\Delta\tilde{N}_k(\mbox{d}u)\right\rangle\\
&-2\frac{1-\theta}{\theta}\left\langle z_{t_k}-D(z_{t_{k-m}}),\int_Uh(y_{t_{k}}, y_{t_{k-m}},u)\Delta\tilde{N}_k(\mbox{d}u)\right\rangle\\
\le&|z_{t_k}-D(z_{t_{k-m}})|^2+\Delta K(1+|y_{t_k}|^2)+\Delta |V(y_{t_{k-m}},0)|^2|y_{t_{k-m}}|^2\\
&+\left|\int_Uh(y_{t_{k}}, y_{t_{k-m}},u)\Delta\tilde{N}_k(\mbox{d}u)\right|^2+\frac{2}{\theta}\left\langle y_{t_k}-D(y_{t_{k-m}}),\int_Uh(y_{t_{k}}, y_{t_{k-m}},u)\Delta\tilde{N}_k(\mbox{d}u)\right\rangle\\
&-2\frac{1-\theta}{\theta}\left\langle z_{t_k}-D(z_{t_{k-m}}),\int_Uh(y_{t_{k}}, y_{t_{k-m}},u)\Delta\tilde{N}_k(\mbox{d}u)\right\rangle.
\end{split}
\end{equation*}
Summing both sides from 0 to $k$, we deduce that
\begin{equation*}
\begin{split}
&|z_{t_{k+1}}-D(z_{t_{k+1-m}})|^2\le|z_{t_0}-D(z_{t_{-m}})|^2+KT+\Delta K\sum\limits_{i=0}^k|y_{t_i}|^2+\Delta \sum\limits_{i=0}^k|V(y_{t_{i-m}},0)|^2|y_{t_{i-m}}|^2\\
&+\sum\limits_{i=0}^k\left|\int_Uh(y_{t_{i}}, y_{t_{i-m}},u)\Delta\tilde{N}_i(\mbox{d}u)\right|^2+\frac{2}{\theta}\sum\limits_{i=0}^k\left\langle y_{t_i}-D(y_{t_{i-m}}),\int_Uh(y_{t_{i}}, y_{t_{i-m}},u)\Delta\tilde{N}_i(\mbox{d}u)\right\rangle\\
&-2\frac{1-\theta}{\theta}\sum\limits_{i=0}^k\left\langle z_{t_i}-D(z_{t_{i-m}}),\int_Uh(y_{t_{i}}, y_{t_{i-m}},u)\Delta\tilde{N}_i(\mbox{d}u)\right\rangle.
\end{split}
\end{equation*}
Consequently,
\begin{equation*}
\begin{split}
&|z_{t_{k+1}}-D(z_{t_{k+1-m}})|^{2p}\le6^{p-1}(|z_{t_0}-D(z_{t_{-m}})|^2+KT)^{p}+6^{p-1}K^p\Delta^p\left(\sum\limits_{i=0}^k|y_{t_i}|^2\right)^{p}\\
&+6^{p-1}\Delta^p\left(\sum\limits_{i=0}^k|V(y_{t_{i-m}},0)|^2|y_{t_{i-m}}|^2\right)^p+6^{p-1}\left(\sum\limits_{i=0}^k\left|\int_Uh(y_{t_{i}}, y_{t_{i-m}},u)\Delta\tilde{N}_i(\mbox{d}u)\right|^2\right)^p\\
&+6^{p-1}4^p\left|\sum\limits_{i=0}^k\left\langle y_{t_i}-D(y_{t_{i-m}}),\int_Uh(y_{t_{i}}, y_{t_{i-m}},u)\Delta\tilde{N}_i(\mbox{d}u)\right\rangle\right|^p\\
&+6^{p-1}2^p\left|\sum\limits_{i=0}^k\left\langle z_{t_i}-D(z_{t_{i-m}}),\int_Uh(y_{t_{i}}, y_{t_{i-m}},u)\Delta\tilde{N}_i(\mbox{d}u)\right\rangle\right|^p.
\end{split}
\end{equation*}
 With assumption (A5), we find that for $0<j<M$,
\begin{equation*}
\begin{split}
&\mathbb{E}\left[\sup\limits_{0\le k\le j}\left(\sum\limits_{i=0}^k\left|\int_Uh(y_{t_{i}}, y_{t_{i-m}},u)\Delta\tilde{N}_i(\mbox{d}u)\right|^2\right)^p\right]\\
\le& M^{p-1}C\mathbb{E}\left(\sum\limits_{i=0}^j\int_U|h(y_{t_{i}}, y_{t_{i-m}},u)|^{2p}\lambda(\mbox{d}u)\right)\\
\le& C\sum\limits_{i=0}^j\mathbb{E}\int_U([1+K_2|y_{t_i}|+V_2(y_{t_{i-m}},0)|y_{t_{i-m}}|]^{2p}|u|^{2pr})\lambda(\mbox{d}u)\\
\le&C+C\sum\limits_{i=0}^j\mathbb{E}|y_{t_i}|^{2p}+C\sum\limits_{i=0}^j\mathbb{E}(|V_2(y_{t_{i-m}},0)|^{2p}|y_{t_{i-m}}|^{2p}).
\end{split}
\end{equation*}
Using (A5), Lemma \ref{usefullemma} and the H\"{o}lder inequality, we compute
\begin{equation*}
\begin{split}
&\mathbb{E}\left[\sup\limits_{0\le k\le j}\left|\sum\limits_{i=0}^k\left\langle y_{t_i}-D(y_{t_{i-m}}),\int_Uh(y_{t_{i}}, y_{t_{i-m}},u)\Delta\tilde{N}_i(\mbox{d}u)\right\rangle\right|^p\right]\\
\le&C\mathbb{E}\left(\sum\limits_{i=0}^j|y_{t_i}-D(y_{t_{i-m}})|^2\int_U|h(y_{t_{i}}, y_{t_{i-m}},u)|^2\lambda(\mbox{d}u)\right)^\frac{p}{2}\\
\le&C\mathbb{E}\sum\limits_{i=0}^j|y_{t_i}-D(y_{t_{i-m}})|^p\int_U[1+K_2|y_{t_i}|+V_2(y_{t_{i-m}},0)|y_{t_{i-m}}|]^{p}|u|^{pr}\lambda(\mbox{d}u)\\
\le&C+C\sum\limits_{i=0}^j\mathbb{E}|y_{t_i}|^{2p}+C\sum\limits_{i=0}^j\mathbb{E}(|V_3(y_{t_{i-m}},0)|^{2p}|y_{t_{i-m}}|^{2p})+C\sum\limits_{i=0}^j\mathbb{E}(|V_2(y_{t_{i-m}},0)|^{2p}|y_{t_{i-m}}|^{2p}).
\end{split}
\end{equation*}
Similarly, by (A5) and Lemma \ref{usefullemma} again
\begin{equation*}
\begin{split}
&\mathbb{E}\left[\sup\limits_{0\le k\le j}\left|\sum\limits_{i=0}^k\left\langle z_{t_i}-D(z_{t_{i-m}}),\int_Uh(y_{t_{i}}, y_{t_{i-m}},u)\Delta\tilde{N}_i(\mbox{d}u)\right\rangle\right|^p\right]\\
\le&C+C\sum\limits_{i=0}^j\mathbb{E}|z_{t_i}-D(z_{t_{i-m}})|^{2p}+C\sum\limits_{i=0}^j\mathbb{E}|y_{t_i}|^{2p}+C\sum\limits_{i=0}^j\mathbb{E}(|V_2(y_{t_{i-m}},0)|^{2p}|y_{t_{i-m}}|^{2p}).
\end{split}
\end{equation*}
This implies that
\begin{equation*}
\begin{split}
&\mathbb{E}\left[\sup\limits_{0\le k\le j+1}|z_{t_{k}}-D(z_{t_{k-m}})|^{2p}\right]\\
\le&C+C\sum\limits_{i=0}^j\mathbb{E}|z_{t_i}-D(z_{t_{i-m}})|^{2p}+C\sum\limits_{i=0}^j\mathbb{E}|y_{t_i}|^{2p}\\
&+C\sum\limits_{i=0}^j\mathbb{E}(|V_2(y_{t_{i-m}},0)|^{2p}|y_{t_{i-m}}|^{2p})+C\sum\limits_{i=0}^j\mathbb{E}(|V_3(y_{t_{i-m}},0)|^{2p}|y_{t_{i-m}}|^{2p})\\
\le&C+C\sum\limits_{i=0}^j\mathbb{E}\left[\sup\limits_{0\le k\le i}|z_{t_k}-D(z_{t_{k-m}})|^{2p}\right]+C\sum\limits_{i=0}^j\mathbb{E}|y_{t_i}|^{2p}+C\sum\limits_{i=0}^j\mathbb{E}|y_{t_{i-m}}|^{2p(l+1)}.
\end{split}
\end{equation*}
By the discrete Gronwall inequality we find that
\begin{equation*}
\begin{split}
\mathbb{E}\left[\sup\limits_{0\le k\le j+1}|z_{t_{k}}-D(z_{t_{k-m}})|^{2p}\right]
\le&C+C\sum\limits_{i=0}^j\mathbb{E}|y_{t_i}|^{2p}+C\sum\limits_{i=0}^j\mathbb{E}|y_{t_{i-m}}|^{2p(l+1)}.
\end{split}
\end{equation*}
Following the steps of \eqref{ztk1-Dztk1}, the desired assertion can be derived by similar skills.
\end{proof}

\subsection{Convergence Rates}
Firstly, we define the corresponding continuous-time split-step $\theta$-EM solution $Z(t)$ as follows: For any $t\in[-\tau,0)$, $Z(t)=\xi(t)$, $Z(0)=\xi(0)-\theta b(\xi(0), \xi(-\tau))\Delta$; For any $t\in[0,T]$,
\begin{equation}\label{jumpcontinuous}
\mbox{d}[Z(t)-D(Z(t-\tau))]=b(\bar{Y}(t), \bar{Y}(t-\tau))\mbox{d}t+\int_U h(\bar{Y}(t), \bar{Y}(t-\tau),u)\tilde{N}(\mbox{d}u,\mbox{d}t),
\end{equation}
where $\bar{Y}(t)$ is defined by
\begin{equation*}
\bar{Y}(t):=y_{t_k} \quad \mbox{for} \quad t\in[t_k, t_{k+1}),
\end{equation*}
thus $\bar{Y}(t-\tau)=y_{t_{k-m}}$. The continuous form of $\theta$-EM solution $Y(t)$ is defined by
\begin{equation}\label{ytztrelation1y}
Y(t)-D(Y(t-\tau))=Z(t)-D(Z(t-\tau))+\theta b(Y(t),Y(t-\tau))\Delta.
\end{equation}
\begin{lem}\label{jumpytytk}
{\rm
Consider the $\theta$-EM scheme \eqref{jumpdiscrete}, and let (A1), (A3)-(A5) hold. Then, for any $p\ge 2$, the continuous form $Y(t)$ of $\theta$-EM scheme has the following properties:
\begin{equation*}
\mathbb{E}\left(\sup\limits_{0\le t\le T}|Y(t)|^p\right)\le C,
\end{equation*}
and
\begin{equation*}
\mathbb{E}\left(\sup\limits_{0\le t\le T}|Y(t)-\bar{Y}(t)|^p\right)\le C\Delta,
\end{equation*}
where $C$ is a constant independent of $\Delta$.
}
\end{lem}
\begin{proof}
The proof is similar to that of Lemma \ref{ytytk}, here we only give the most critical part to show the differences between the Brownian motion case.
For $t\in[t_k, t_{k+1})$, \eqref{jumpcontinuous} gives that
\begin{equation*}
\begin{split}
&Z(t)-D(Z(t-\tau))-Z(t_k)+D(Z(t_{k-m}))\\
=&\int_{t_k}^tb(\bar{Y}(s), \bar{Y}(s-\tau))\mbox{d}s+\int_{t_k}^t\int_Uh(\bar{Y}(s), \bar{Y}(s-\tau),u)\tilde{N}(\mbox{d}u,\mbox{d}s).
\end{split}
\end{equation*}
Denote by $\Phi(Z(t),Z(t_k))=Z(t)-D(Z(t-\tau))-Z(t_k)+D(Z(t_{k-m}))$, then
\begin{equation*}
\begin{split}
\mathbb{E}\left(\sup\limits_{t_k\le t<t_{k+1}}|\Phi(Z(t),Z(t_k))|^p\right)\le &2^{p-1}\mathbb{E}\left(\sup\limits_{t_k\le t<t_{k+1}}\left|\int_{t_k}^t b(\bar{Y}(s), \bar{Y}(s-\tau))\mbox{d}s\right|^p\right)\\
&+2^{p-1}\mathbb{E}\left(\sup\limits_{t_k\le t<t_{k+1}}\left|\int_{t_k}^t\int_Uh(\bar{Y}(s), \bar{Y}(s-\tau),u)\tilde{N}(\mbox{d}u,\mbox{d}s)\right|^p\right).
\end{split}
\end{equation*}
Application of (A4), Lemmas \ref{usefullemma}-\ref{jumppmoment}, and the H\"{o}lder inequality give that
\begin{equation*}
\begin{split}
\mathbb{E}\left(\sup\limits_{t_k\le t<t_{k+1}}|\Phi(Z(t),Z(t_k))|^p\right)
\le& 2^{p-1}\Delta^{p-1}\mathbb{E}\int_{t_k}^{t_{k+1}}\left|b(\bar{Y}(s), \bar{Y}(s-\tau))\right|^p\mbox{d}s\\
&+C\mathbb{E}\int_{t_k}^{t_{k+1}}\int_U|h(\bar{Y}(s), \bar{Y}(s-\tau),u)|^p\lambda(\mbox{d}u)\mbox{d}s\\
\le &C\Delta^{p}+C\Delta\le C\Delta.
\end{split}
\end{equation*}
Following the proof of Lemma \ref{ytytk}, we shall get the desired result.
\end{proof}

\begin{thm}\label{theorem2}
{\rm Let assumptions (A1), (A3)-(A5) hold, then the $\theta$-EM solution $Y(t)$ converges to the exact solution $X(t)$ in $\mathcal{L}^p$ sense, i.e.,
\begin{equation*}
\mathbb{E}\left(\sup\limits_{0\le t\le T}|Y(t)-X(t)|^p\right)\le C\Delta^{\frac{1}{2}}
\end{equation*}
for $p\ge 2$.
}
\end{thm}
\begin{proof}
Let $e(t)=Z(t)-D(Z(t-\tau))-X(t)+D(X(t-\tau))$, it is obvious that
\begin{equation*}
\begin{split}
e(t)=&e(0)+\int_0^t[b(\bar{Y}(s),\bar{Y}(s-\tau))-b(X(s),X(s-\tau))]\mbox{d}s\\
&+\int_0^t\int_U[h(\bar{Y}(s),\bar{Y}(s-\tau),u)-h(X(s),X(s-\tau),u)]\tilde{N}(\mbox{d}u,\mbox{d}s),
\end{split}
\end{equation*}
where $e(0)=-\theta b(\xi(0), \xi(-\tau))\Delta$. Define
\begin{equation*}
\begin{split}
\mu(t)=b(\bar{Y}(t),\bar{Y}(t-\tau))-b(X(t),X(t-\tau)),
\end{split}
\end{equation*}
and
\begin{equation*}
\begin{split}
\upsilon(t)=h(\bar{Y}(t),\bar{Y}(t-\tau),u)-h(X(t),X(t-\tau),u).
\end{split}
\end{equation*}
Application of the It\^{o} formula yields
\begin{equation*}
\begin{split}
|e(t)|^p=&|e(0)|^p+p\int_0^t |e(s)|^{p-2}\langle e(s),\mu(s)\rangle\mbox{d}s\\
&+\int_0^{t}\int_U[|e(s)+\upsilon(s)|^p-|e(s)|^p-p|e(s)|^{p-2}\langle e(s),\upsilon(s)\rangle]\lambda(\mbox{d}u)\mbox{d}s\\
&+\int_0^{t}\int_U[|e(s)+\upsilon(s)|^p-|e(s)|^p]\tilde{N}(\mbox{d}u,\mbox{d}s)\\
\le&|e(0)|^p+p\int_0^t |e(s)|^{p-2}\langle e(s),\mu(s)\rangle\mbox{d}s+C\int_0^t\int_U|e(s)|^{p-2}|\upsilon(s)|^2\lambda(\mbox{d}u)\mbox{d}s\\
&+C\int_0^t\int_U|\upsilon(s)|^p\lambda(\mbox{d}u)\mbox{d}s+\int_0^{t}\int_U[|e(s)+\upsilon(s)|^p-|e(s)|^p]\tilde{N}(\mbox{d}u,\mbox{d}s)\\
=:&|e(0)|^p+\bar{H}_1(t)+\bar{H}_2(t)+\bar{H}_3(t)+\bar{H}_4(t).
\end{split}
\end{equation*}
Similar to the derivation of Theorem \ref{theorem1}, with (A5) and  Lemmas \ref{jumppmoment}-\ref{jumpytytk}, we calculates
\begin{equation*}
\begin{split}
&\mathbb{E}\left(\sup\limits_{0\le u\le t}|\bar{H}_1(u)|\right)\\
\le& C\mathbb{E}\int_0^t|e(s)|^p\mbox{d}s+C\mathbb{E}\int_0^t|b(\bar{Y}(s),\bar{Y}(s-\tau))-b(Y(s),\bar{Y}(s-\tau))|^p\mbox{d}s\\
&+C\mathbb{E}\int_0^t|b(Y(s),\bar{Y}(s-\tau))-b(Y(s),Y(s-\tau))|^p\mbox{d}s\\
&+C\mathbb{E}\int_0^t|b(Y(s),Y(s-\tau))-b(X(s),X(s-\tau))|^p\mbox{d}s\\
\le& C\int_0^t\mathbb{E}\left(\sup\limits_{0\le u\le s}|Y(u)-X(u)|^p\right)\mbox{d}s+C\Delta^{p}\\
&+C\int_0^t[\mathbb{E}(1+|\bar{Y}(s)|^{l_1}+|Y(s)|^{l_1})^{2p}]^{\frac{1}{2}}[\mathbb{E}|\bar{Y}(s)-Y(s)|^{2p}]^{\frac{1}{2}}\mbox{d}s\\
&+C\int_0^t[\mathbb{E}(1+|\bar{Y}(s-\tau)|^{l_1}+|Y(s-\tau)|^{l_1})^{2p}]^{\frac{1}{2}}[\mathbb{E}|\bar{Y}(s-\tau)-Y(s-\tau)|^{2p}]^{\frac{1}{2}}\mbox{d}s\\
&+C\int_0^t[\mathbb{E}(1+|Y(s)|^{l_1}+|X(s)|^{l_1})^{2p}]^{\frac{1}{2}}[\mathbb{E}|Y(s)-X(s)|^{2p}]^{\frac{1}{2}}\mbox{d}s\\
&+C\int_0^t[\mathbb{E}(1+|Y(s-\tau)|^{l_1}+|X(s-\tau)|^{l_1})^{2p}]^{\frac{1}{2}}[\mathbb{E}|Y(s-\tau)-X(s-\tau)|^{2p}]^{\frac{1}{2}}\mbox{d}s\\
\le& C\int_0^t\mathbb{E}\left(\sup\limits_{0\le u\le s}|Y(u)-X(u)|^p\right)\mbox{d}s+C\Delta^{p}+C\Delta^{\frac{1}{2}}.
\end{split}
\end{equation*}
Similarly, we obtain
\begin{equation*}
\begin{split}
&\mathbb{E}\left(\sup\limits_{0\le u\le t}|\bar{H}_2(u)|\right)+\mathbb{E}\left(\sup\limits_{0\le u\le t}|\bar{H}_3(u)|\right)\\
\le&C\int_0^t\mathbb{E}\left(\sup\limits_{0\le u\le s}|Y(u)-X(u)|^p\right)\mbox{d}s+C\Delta^{p}+C\Delta^{\frac{1}{2}}.
\end{split}
\end{equation*}
Furthermore, by Lemmas \ref{usefullemma}-\ref{jumpytytk} and the H\"{o}lder inequality, we compute
\begin{equation*}
\begin{split}
\mathbb{E}\left(\sup\limits_{0\le u\le t}|\bar{H}_4(u)|\right)\le&\frac{1}{4}\mathbb{E}\left(\sup\limits_{0\le u\le t}|e(u)|^p\right)+C\mathbb{E}\left(\int_0^t\int_U|\upsilon(s)|^p \lambda(\mbox{d}u)\mbox{d}s\right)\\
\le&\frac{1}{4}\mathbb{E}\left(\sup\limits_{0\le u\le t}|e(u)|^p\right)+C\int_0^t\mathbb{E}\left(\sup\limits_{0\le u\le s}|Y(u)-X(u)|^p\right)\mbox{d}s+C\Delta^{p}+C\Delta^{\frac{1}{2}}.
\end{split}
\end{equation*}
Putting $\bar{H}_1(t)-\bar{H}_4(t)$ together, we arrive at
\begin{equation*}
\begin{split}
\mathbb{E}\left(\sup\limits_{0\le u\le t}|e(u)|^p\right)\le C\int_0^t\mathbb{E}\left(\sup\limits_{0\le u\le s}|Y(u)-X(u)|^p\right)\mbox{d}s+C\Delta^{\frac{1}{2}}.
\end{split}
\end{equation*}
Consequently, following the process of Theorem \ref{theorem1}, the desired result will be obtained.

\end{proof}

\begin{rem}
We see from Theorems \ref{theorem1} and \ref{theorem2} that the strong convergence rate of $\theta$-EM scheme for neutral SDDEs is $\frac{1}{2}$ for the Brownian motion case, while for the pure jumps case, the order is $\frac{1}{2p}$, that is to say, lower moment has a better convergence rate for neutral SDDEs with jumps, whence it is better to use the mean-square convergence for jump case.

\end{rem}

\begin{thm}
{\rm Let (A1), (A3)-(A5) hold, then the continuous form of $\theta$-EM scheme \eqref{jumpdiscrete} converges to the exact solution of \eqref{jump} almost surely with order $\alpha<\frac{1}{2p}$, i.e., there exists a finite random variable $\zeta_\alpha$ such that
\begin{equation*}
\sup\limits_{0\le t\le T}|Y(t)-X(t)|\le\zeta_\alpha\Delta^\alpha
\end{equation*}
for $\alpha\in(0,\frac{1}{2p})$.
}
\end{thm}
\begin{proof}
The desired result can be obtained with Theorem \ref{theorem2} similar to the process of Theorem \ref{astheorem}.
\end{proof}

\end{document}